\documentclass[12pt]{article}
\usepackage{amssymb}
\usepackage{amsmath}
\usepackage{mathtools}
\usepackage[utf8]{inputenc}
\usepackage[breaklinks=true]{hyperref}
\usepackage{comment}
\usepackage{mathrsfs}
\usepackage{geometry}
\usepackage{amsthm}
\usepackage{amsrefs}
\usepackage{float}
\usepackage{xcolor}
\numberwithin{equation}{section}

\newtheorem{theorem}{Theorem}[section]

\newtheorem{lemma}[theorem]{Lemma}

\newcommand{\E}{\mathbb{E}}
\renewcommand{\P}{\mathbb{P}}

\newcommand{\T}{\mathbb{T}}

\newcommand{\R}{\mathbb{R}}
\newcommand{\La}{(\textbf{L1})}
\newcommand{\Lb}{(\textbf{L2})}

\newcommand{\Hg}{(\textbf{H})}

\title{Branching random walks with regularly varying perturbations}
\author{Krzysztof Kowalski\footnote{University of Wrocław, Poland. Email: krzysztof.kowalski@math.uni.wroc.pl. Supported by the National Science Center, Poland (OPUS, grant number 2019/33/B/ST1/00207)
\\ MSC2020 subject classification: 60J80, 60F05, 60F15
\\ keywords: branching random walk, perturbations, maximal position, central limit theorem, strong limit, derivative martingale, additive martingale, heavy-tailed distributions}}

\begin{document}

\maketitle

\begin{abstract}
    We consider a modification of classical branching random walk, where we add i.i.d. perturbations to the positions of the particles in each generation. In this model, which was introduced and studied by Bandyopadhyay and Ghosh (2023), perturbations take form $\frac 1 \theta \log\frac X E$, where $\theta$ is a positive parameter, $X$ has arbitrary distribution $\mu$ and $E$ is exponential with parameter 1, independent of $X$. Working under finite mean assumption for $\mu$, they proved almost sure convergence of the rightmost position to a constant limit, and identified the weak centered asymptotics when $\theta$ does not exceed certain critical parameter $\theta_0$. This paper complements their work by providing weak centered asymptotics for the case when $\theta > \theta_0$ and extending the results to $\mu$ with regularly varying tails. We prove almost sure convergence of the rightmost position and identify the appropriate centering for the weak convergence, which is of form $\alpha n + c \log n$, with constants $\alpha$, $c$ depending on the ratio of $\theta$ and $\theta_0$. We describe the limiting distribution and provide explicitly the constants appearing in the centering.
\end{abstract}
\section{Introduction}

Branching random walk on the real line, abbreviated as BRW, is constructed as follows. The process starts with a  single particle placed  at 0. 
Given a point process ${\mathcal Z} = \sum_{k=1}^N \delta_{\xi_k}$ on $\R$, where $N$ is a random variable on $\mathbb{N}_0$, the original particle at time 1 dies and gives birth to $N$ particles positioned according to ${\mathcal Z}$. 
These particles are called the first generation of the process. At time 2, each of these particles reproduces independently 
and has offspring with positions relative to their parent’s position given by an independent copy of ${\mathcal Z}$. 

The process continues infinitely. As a result we obtain a marked tree $\left(S, \mathbb{T}\right)$, where the tree $\mathbb{T}$ is the set of all particles  equipped in the natural tree structure, and $S_v$ is the position of a given particle $v \in \mathbb{T}$. We write $|v|$ for the generation of $v$. For a BRW with displacements given by $\xi$,
let $R_n = \sup_{|v|=n} S_v$ denote the position of the most right particle at time $n$. There is a rich literature concerning the asymptotic behaviour of $R_n$. In 1976 Biggins \cite{biggins2} proved under minimal assumptions the law of large numbers for $R_n$,  i.e.  $\frac {R_n} n$ converges almost surely to a constant $\alpha$. The corresponding limit theorem was proved by Aïdékon \cite{aidekon} in 2013, who showed that $R_n - \alpha n + c \log n $ converges in distribution to a random shift of the Gumbel distribution.
We refer to Shi \cite{brw} for an extensive description of recent results on branching random walks.

In this paper we consider a perturbed branching random walk $S^*$, which is a modification of $S$, in which we add a random perturbation
to the position of every particle,
i.e.
\[
S^*_v = S_v + X_v
\]
where $\{X_v\}_{v \in \mathbb{T}}$ are i.i.d. random variables independent of $S$. Note that the perturbation added to the position of a vertex $v \in \T$ does not influence the positions of its offspring, which explains that the process is sometimes called \textit{last progeny modified branching random walk}. In this paper we study the model introduced by Bandyopadhyay and Ghosh in \cite{lpm-brw}, where the perturbations have the form
\[
X_{v}(\theta) = \frac 1 \theta \log \frac{Y_v} {E_v}
\]
for a given positive real number $\theta$, and $\{Y_v\}_{v \in \mathbb{T}}$ that are independent positive random variables with distribution $\mu$, independent of $\{E_v\}_{v \in \mathbb{T}}$, which are i.i.d. with distribution $\text{Exp}{\left(1\right)}$. This model was further studied in the context of large deviations \cite{large-deviations} and inhomogeneous time BRW \cite{inhm-time}. Main motivation for considering it comes from the connection between the supremum of the perturbed BRW $R^*_n
    (\theta) = \sup_{|v|=n} S^*_v$  and random weighted sums. More precisely,  Theorem 3.6 in \cite{lpm-brw} states, that
\begin{align}\label{coupling}
\theta R^*_n(\theta) \stackrel{d}{=} \log Y_n(\theta) -\log E 
\end{align}
where $Y_n(\theta) = \sum_{|v|=n} e^{\theta S_v} Y_v$ and $E$ is exponential with parameter 1, independent of $Y_n(\theta)$. We will sometimes write $R^*_n = R^*_n(\theta)$ and $Y_n = Y_n(\theta)$ if the parameter is clear from the context.  The asymptotics of $R^*_n$ will be related very closely to the behaviour of $Y_n$. As it turns out, the sequence $Y_n$ is quite well described in the literature, see e.g. \cite{boundary} and \cite{above}.  
Properties of $R_n^*$ are obviously determined by the parameter $\theta$. It turns out that one needs to control its position with respect to the critical parameter $\theta_0$ defined as 
\[
\theta_0 = \inf \left\{\theta > 0 : \nu(\theta) = \theta \nu'(\theta)\right\}
\]
where 
\[\nu (\theta) = \log \mathbb{E} \left[\sum_{i=1}^N e^{\theta \xi_i}\right] \]
is the log-Laplace transform of $\mathcal{Z}$, and
$\nu'(\theta) =  e^{- \nu(\theta)}\E \left[ \sum_{i=1}^N \xi_i e^{\theta \xi_i} \right]$. Note that $\nu$ does not have to be differentiable at $\theta$ for this quantity to exist, and that in general $\theta_0$ may be infinite. 

In \cite{lpm-brw} branching random walks with such perturbations where studied in the case when $\mu$ has finite mean. In particular, the authors proved that 

\begin{align*}
 \frac {R_n^*(\theta)} n \xrightarrow[ n \rightarrow \infty]{a.s.} \begin{cases} \frac {\nu\left(\theta\right)} {\theta} &  \theta < \theta_0 \\
 \frac {\nu(\theta_0)} {\theta_0} & \theta \geq \theta_0
\end{cases}
\end{align*}
 and identified weak centered asymptotics for $\theta \leq \theta_0$. However, the result for $\theta > \theta_0$ was only obtained for the degenerated perturbations with $\mu = \delta_1$. The goal of this paper is to complete the results from \cite{lpm-brw} by providing the missing weak centered asymptotics for the so called above the boundary case, and to extend them to $\mu$ with infinite mean, with special focus on distributions with regularly varying tails.

\section{Main results} 
Let $\gamma \in \left(0,1\right)$. Our main assumption for $\mu$ is as follows:

\begin{equation}
    \tag{\textbf{H}} x^\gamma \left(1 - F(x)\right) \xrightarrow[ x\rightarrow +\infty]{} c_+ >0 
\end{equation} 
where $F$ is the probability distribution function of $\mu$.

This assumption tells us that $\mu$ belongs to the domain of attraction of a stable law with characteristic function
\begin{equation}\label{g}
\tilde{g}\left(t\right) = e^{-k|t|^\gamma  \left(1 - i \tan{\left(\frac {\pi \gamma} {2}\right)} \text{sign}t\right)}    
\end{equation}

where
\[
k = \frac {\pi c_+} {2 \Gamma\left(\gamma\right) \sin\left(\pi \gamma /2\right)} > 0.
\]
Furthermore, it means that if $Y$ has distribution $\mu$, then $\E\left[Y^\gamma\right] = \infty$, but $\E\left[Y^r\right] < \infty$ for any $r \in \left(0, \gamma\right)$. $\Hg$  will be assumed in most of the paper, however the result for the above the boundary case will be stated under more general assumption.

We assume for the rest of the paper that $\text{supp}\left(\mu\right) \subset \mathbb{R}_+$, the system survives with probability 1 (i.e. $\P (N=0) =0$), and $\nu$ is finite on some open interval $I$, with $0 \in I$. Since $\theta_0$ is finite, the last assumption guarantees, by convexity of $\nu$, that $\nu$ is differentiable on $\left(-s,\theta_0\right)$ for some $s>0$, and has the left derivative at $\theta_0$. One can also characterize $\theta_0$ as the argument minimizing $\frac {\nu(\theta)} \theta$ over $\theta > 0$. Throughout this paper, existence of finite $\theta_0$ will only be assumed when necessary. 

As proved in \cite{biggins2}, if $\theta_0$ is finite, then

\begin{align}\label{a.s. convergence of the supremum}
\frac{R_n}{n} \xrightarrow[n \rightarrow \infty]{} \frac {\nu(\theta_0)} {\theta_0}, \quad \mbox{a.s.}
\end{align}

For $\theta$ such that $\nu(\theta) < \infty$, let 
\begin{align}
W_n(\theta) = e^{-n\nu(\theta)}\sum_{|v|=n} e^{\theta S_v}.
\end{align}
$W_n(\theta)$ is called the additive martingale associated with $S$. We denote $W_n = W_n(\theta_0)$. Note that as a positive martingale $W_n(\theta)$ converges almost surely to some finite limit. If $\nu'(\theta) < \infty$, then Biggins martingale convergence theorem \cite{biggins} states, that the almost sure limit of $W_n(\theta)$ is non-degenerate if and only if $\nu'(\theta) <  \nu(\theta)/\theta$ and  \begin{equation}\label{eq:h2} 
  \E \left[W_1(\theta) \log_+ W_1(\theta)\right] < \infty. 
  \end{equation} Furthermore, the limit is then positive almost surely. The first condition is equivalent to $\theta < \theta_0$, thus
\begin{align}\label{additive martingale limit}
 W_n(\theta) \xrightarrow[ n \rightarrow \infty]{a.s.} \begin{cases} W^\infty_\theta & \mbox{if $\theta < \theta_0$  and \eqref{eq:h2} is satisfied,}\\
0 & \mbox{otherwise.}
\end{cases}
\end{align}
where $W_{\theta}^\infty$ is finite and positive almost surely. We also define the \textbf{derivative martingale} associated with $S$ as
\[
D_n = -\sum_{|v|=n} \left(\theta_0 S_v - n \nu(\theta_0) \right) e^{\theta_0 S_v - n\nu(\theta_0)}
\]

As seen in Proposition A.3 from \cite{aidekon}, under assumptions 
\begin{description}
\item[(L1)] $\theta_0 < \infty$ and \begin{align*}\E \left[\sum_{i=1}^N e^{\theta_0 \xi_i}\xi_i^2\right] < \infty. \end{align*}

\item[(L2)]
$\theta_0 < \infty$, and for $\tilde{X} = \sum_{i=1}^N e^{\theta_0 \xi_i} \xi_i$, $X = \sum_{i=1}^N  e^{\theta_0\xi_i}$
 \begin{align*}
  &\E \left[\tilde{X}\log_+\tilde{X}\right] <\infty, \\
 &\E \left[X\log_+^2 X\right] < \infty, 
 \end{align*}
 \end{description}
 where $\log_+x = \max\{0,\log x\}$ and $\log^2_+x = \left(\log_+ x\right)^2$, we have
\begin{align}\label{derivative martingale limit}
    D_n \xrightarrow[ n \rightarrow \infty]{a.s.} D_\infty
\end{align}
where $D_\infty$ is finite and positive almost surely. 

These two martingales are connected through Theorem 1.1 from \cite{aidekon-shi}, which states that under $\La$ and $\Lb$
\begin{align}\label{convergence of normalized additive martingale}
n^{\frac 1 2} W_n(\theta_0) \xrightarrow[n \rightarrow \infty]{\mathbb{P}} c_\infty D_\infty 
\end{align}

where
\[ c_\infty = \left( \frac 2 {\pi \sigma^2} \right)^{\frac 1 2} \quad \mbox{and} \quad
\sigma^2 = \E\left[ \sum_{i=1}^N \left(\theta_0 \xi_i -  \nu(\theta_0)\right)^2 e^{\theta_0 \xi_i - \nu(\theta_0)} \right].
\]
For more results on these martingales and their limits see for example Chapter 3 in \cite{brw}.

We are now ready to present our results, starting with the almost sure convergence.

\begin{theorem}\label{a.s. convergence below}
Assume that condition $\Hg$ is satisfied and $\E \left[W_1(\gamma \theta) \log_+ W_1(\gamma \theta)\right] < \infty$ for $\theta < \frac {\theta_0} \gamma$. Then

\begin{align*}
 \frac {R_n^*} n \xrightarrow[ n \rightarrow \infty]{a.s.} \frac {\nu\left(\gamma \theta\right)} {\gamma \theta}.
\end{align*}

\end{theorem}

\begin{theorem}\label{a.s. convergence above} If $\theta_0 \leq  \theta$ and $\mu$ has finite r-th moments for all $r < \frac {\theta_0} {\theta}$, then 
\begin{align*}
   \frac {R_n^*} n \xrightarrow[ n \rightarrow \infty]{a.s.}  \frac {\nu(\theta_0)} {\theta_0}. 
\end{align*}
\end{theorem}
In particular, the conditions of Theorem \ref{a.s. convergence above} hold if $\mu$ satisfies $\Hg$ and $ \theta \geq \frac {\theta_0} \gamma$.

The results concerning convergence in distribution we split into three cases.

\begin{theorem}[Below the boundary case]\label{convergence in distribution below the boundary}
Assume that condition $\Hg$ is satisfied and $\E \left[W_1(\gamma \theta) \log_+ W_1(\gamma \theta)\right] < \infty$ for $\theta <  \frac {\theta_0} \gamma$. Then
\[
R^*_n - n \frac {\nu \left(\gamma \theta\right)} {\gamma \theta} \xrightarrow[n \rightarrow \infty]{d} \frac 1 \theta \left(\log H_\theta - \log E\right),
\]
where $H_\theta$ is finite and positive almost surely, and $E$ is exponential with intensity 1, independent of $H_\theta$. Furthermore, $H_\theta$ has the characteristic function $\mathbb{E}\left[\tilde{g}\left(t\left(W^{\infty}_{\gamma \theta}\right)^\gamma\right)\right]$ where $W^\infty_{\gamma \theta}$ is the limit from (\ref{additive martingale limit}) and $\tilde{g}$ is defined in \eqref{g}.
\end{theorem}

\begin{theorem}[The boundary case]\label{convergence in distribution - boundary case}
Assume $\La$ and $\Lb$.
 If $\mu$ satisfies $\Hg$, then for $\theta = \frac {\theta_0} \gamma$
\[
R^*_n - n \frac {\nu \left(\theta_0\right)} {\theta_0} + \frac 1 {2 \theta_0} \log n \xrightarrow[n \rightarrow \infty]{d} \frac 1 \theta \left(\log H_{\theta_0} - \log E\right)
\]
where $H_{\theta_0}$ is finite and positive almost surely and $E$ is exponential with intensity 1, independent of $H_{\theta_0}$. Furthermore, $H_{\theta_0}$ has the characteristic function $\mathbb{E}\left[\tilde{g}\left(t (c_\infty D_\infty)^\gamma\right)\right]$, where
$\tilde{g}$ is defined in \eqref{g}.

    \end{theorem}
\begin{theorem}[Above the boundary case]\label{convergence in distribution above the boundary}
Assume $\La$, $\Lb$ and that for all $s \in \mathbb{R}$, $\P\left(\xi_1, \xi_2, \dots, \in s\mathbb{Z}\right) < 1$.
  If $\theta > \theta_0$ and $\mu$ has a finite $r$-th moment for some $r > \frac {\theta_0} {\theta}$ and is not concentrated on a single point, then
\[
R_n^* - n \frac {\nu(\theta_0)}{\theta_0} + \frac {3\log n} {2 \theta_0} \xrightarrow[n \rightarrow \infty]{d} \frac 1 \theta \left(\log Z - \log E\right)
\]
where $Z$ is finite and positive almost surely and $E$ is exponential with intensity 1, independent of $Z$. Furthermore, $Z \stackrel{d}{=} (D_\infty)^{ \frac \theta {\theta_0}} Z_{\frac {\theta_0} {\theta} }$, where $D_\infty$ is the almost sure limit of the derivative martingale defined in \eqref{derivative martingale limit} and $Z_{\frac {\theta_0} {\theta}}$ is strictly $\frac {\theta_0} {\theta}$-stable independent of $D_\infty$. 
\end{theorem}
In particular, if $\mu$ has finite mean and $\theta > \theta_0$, or if $\mu$ satisfies $\Hg$ and $\theta > \frac {\theta_0} \gamma$, then the assumption in the last theorem is satisfied. It is worth noting that the logarithmic correction term in Theorem \ref{convergence in distribution above the boundary} is the same as in the case of a classic BRW without perturbations (see \cite{aidekon}). This result is also more extensive than Theorem 2.6 in \cite{lpm-brw}, where the asymptotics were only given for the case $\mu = \delta_1$.

\section{Proofs of Theorems \ref{convergence in distribution below the boundary}, \ref{convergence in distribution - boundary case}, and \ref{convergence in distribution above the boundary}}

First we recall Lemma 4.1 from \cite{boundary} (presented here in a slightly more accessible form for our use), that will be useful to understand behaviour of the asymptotics of $Y_n$.  
\begin{lemma}\label{convergence of random sums - below the boundary}
Let $\{Y_v\}_{v \in \T }$ be i.i.d. random variables with distribution $\mu$ satisfying $\Hg$, and $\{A_v\}_{v \in \T }$ be a sequence of positive random variables, independent of $\{Y_v\}_{v \in \T }$, such that 

\[
\sum_{|v|=n} A_v^\gamma \xrightarrow[n \rightarrow \infty]{\P} A \quad \mbox{and}\quad \sup_{|v|=n} A_v \xrightarrow[n \rightarrow \infty]{\mathbb{P}} 0
\]
for some positive random variable $A$.
Then
\[
\sum_{|v|=n} A_v Y_v \xrightarrow[n \rightarrow \infty]{d} H
\]
where $H$ has the characteristic function $ \varphi_H\left(t\right) = \mathbb{E}\left[\tilde{g}(tA^{\frac 1 \gamma})\right]$.

\end{lemma}

\begin{proof}[Proof of Theorem \ref{convergence in distribution below the boundary}]
First we will prove that
\begin{equation}\label{eq:h3}
Y_n(\theta) e^{-n \frac {\nu\left(\gamma \theta\right)}{\gamma}} \xrightarrow[n \rightarrow \infty]{d} H_\theta
\end{equation}
where $H_\theta$ has the characteristic function $\mathbb{E}\tilde{g}\big(t(W^\infty_{\gamma \theta})^{\frac 1 \gamma}\big)$ and moreover $H_\theta$ is positive almost surely.
For this purpose we will use Lemma \ref{convergence of random sums - below the boundary} for $A_v = e^{\theta S_v - |v|\frac {\nu\left(\theta \gamma\right)} {\gamma}}$. To check its hypotheses observe
\[
\sup_{|v|=n} A_v = \sup_{|v|=n} e^{n \theta \left(\frac  {S_v} n -  \frac {\nu\left(\theta \gamma\right)} { \theta \gamma}\right)} = e^{n \theta \left(  \frac{R_n}n -  \frac {\nu\left(\theta \gamma\right)} { \theta \gamma}\right)}. 
\]
As $\frac{R_n}{n}  \xrightarrow[n \rightarrow \infty]{a.s.} \frac {\nu(\theta_0)} { \theta_0}$ and $\theta_0$ is the argument minimizing $\frac {\nu\left(t\right)} { t}$, then $n \theta \left(  \frac{R_n}{n} -  \frac {\nu\left(\theta \gamma\right)} { \theta \gamma}\right) \xrightarrow[n \rightarrow \infty]{a.s.}  -\infty$ and so $ \sup_{|v|=n} A_v \xrightarrow[n \rightarrow \infty]{a.s.}  0$. Furthermore, in view of \eqref{additive martingale limit},
\[
\sum_{|v|=n} A_v^\gamma = W_n\left(\theta \gamma\right)
\]
converges to $W^\infty_{\gamma \theta}$, which is positive almost surely. Summarizing,  Lemma \ref{convergence of random sums - below the boundary} 
entails \eqref{eq:h3}

To see that $H_\theta$ is positive almost surely, choose any $\varepsilon > 0$. Then using Exercise 3.3.2 from \cite{durrett}
\begin{align*}
    \mathbb{P}\left(H_\theta = 0\right) &= \lim_{T \rightarrow \infty} \frac 1 {2T} \int_{-T}^T \mathbb{E}\left[\tilde{g}\big(t(W^\infty_{\gamma \theta})^{\frac 1 \gamma}\big) \right]dt  
    \\ &\leq \lim_{T \rightarrow \infty} \frac 1 {2T} \int_{-T}^T \mathbb{E} \left[e^{-k|t|^\gamma  W^\infty_{\gamma \theta}}\right]dt 
    \\&= \lim_{T \rightarrow \infty} \frac 1 {2T} \int_{-T}^T \left( \mathbb{E}\left[ e^{-k|t|^\gamma | W^\infty_{\gamma \theta}|} ; W^\infty_{\gamma \theta} \leq \varepsilon \right] + \mathbb{E} \left[e^{-k|t|^\gamma  W^\infty_{\gamma \theta}}; W^\infty_{\gamma \theta} > \varepsilon\right] \right)dt 
    \\&\leq \mathbb{P}\left(W^\infty_{\gamma \theta} \leq \varepsilon\right) + \lim_{T \rightarrow \infty} \frac 1 {2T} \int_{-T}^T e^{-k|t|^\gamma \varepsilon} dt
\end{align*}
Now the function $t \rightarrow e^{-k \varepsilon |t|^{\gamma}}$ is integrable, hence the limit of the second term is 0 for any $\varepsilon$. The first term can be made arbitrarily small through the choice of $\varepsilon$. Since we know that $W^\infty_{\gamma \theta}$ is positive almost surely, we conclude positivity of $H_\theta$.

Next, recalling \eqref{coupling}
\[
R^*_n - n \frac {\nu \left(\gamma \theta\right)} {\gamma \theta} \stackrel{d}{=} \frac 1 \theta \left( \log Y_n(\theta) - \log E \right) - n \frac {\nu \left(\gamma \theta\right)} {\gamma \theta} = \frac 1 \theta \left( \log Y_n(\theta) e^{-n \frac {\nu\left(\gamma \theta\right)} {\gamma} } - \log E \right)
\]
with $E \sim Exp\left(1\right)$ independent of $Y_n$. Finally, by \eqref{eq:h3} and the continuous mapping theorem 
\[
\log \left(Y_n(\theta) e^{-n \frac {\nu\left(\gamma \theta\right)}{\gamma}}\right) \xrightarrow[n \rightarrow \infty]{d} \log H_\theta
\]
where distribution of $H_\theta$ is as specified in the statement.
\end{proof}

\begin{proof}[Proof of Theorem \ref{convergence in distribution - boundary case}(the boundary case)]
Define $A_v = e^{ \theta S_v - n \frac {\nu \left(\theta_0\right) \theta} {\theta_0} + \frac {\theta} {2 \theta_0} \log n}$ then, by \eqref{convergence of normalized additive martingale}, we have 
\[
\sum_{|v| = n} A^\gamma_v = n^{\frac 1 2} W_n \xrightarrow[n \rightarrow \infty]{\mathbb{P}} \left( \frac 2 {\pi \sigma^2} \right)^{\frac 1 2} D_\infty = Z_{\theta_0}
\]

The second condition of Lemma \ref{convergence of random sums - below the boundary}, $\sup_{|v|=n} A_v \xrightarrow[n \rightarrow \infty]{\mathbb{P}} 0$, follows by applying Proposition A.3. in \cite{iksanov} to $V_u = \theta_0 S_u - |u| \nu(\theta_0)$.
Then once again by Lemma \ref{convergence of random sums - below the boundary} 
\begin{equation}\label{eq:nn1}
Y_n(\theta) n^{\frac {\theta} {2 \theta_0}}e^{-n \frac {\theta \nu(\theta_0)}{\theta_0}} = \sum_{|v| = n} A_v Y_v \xrightarrow[n \rightarrow \infty]{d} H_{\theta_0}.
\end{equation}
Next, by \eqref{coupling}
\begin{align*}
&R^*_n - n \frac {\nu \left(\theta_0\right)} {\theta_0} + \frac 1 {2 \theta_0} \log n \stackrel{d}{=} \frac 1 \theta \left( \log Y_n(\theta) - \log E \right) - n \frac {\nu \left(\theta_0\right)} {\theta_0} + \frac 1 {2 \theta_0} \log n\\&= \frac 1 \theta \left( \log \big( n^{\frac {\theta} {2 \theta_0} }Y_n(\theta) e^{-n \frac {\theta \nu(\theta_0)} {{\theta_0}} } \big) - \log E \right) 
\end{align*}
and by \eqref{eq:nn1} and the continuous mapping theorem 
\[
\log \big( n^{\frac {\theta} {2 \theta_0} }Y_n(\theta) e^{-n \frac {\theta \nu(\theta_0)} {{\theta_0}} } \big) \xrightarrow[n \rightarrow \infty]{d} \log H_{\theta_0}
\]
where distribution of $H_{\theta_0}$ is as specified in the statement.
\end{proof}

\begin{proof}[Proof of Theorem \ref{convergence in distribution above the boundary} (above the boundary case)]

The proof relies on Proposition 3.2 in \cite{above}. The assumptions for Theorem \ref{convergence in distribution above the boundary} with condition $\theta_0 \nu'(\theta_0) = \nu(\theta_0)$ for $S$ are equivalent to assumptions \textbf{(A1)} through  \textbf{(A3)} from \cite{above} for a BRW
\[
V_u= -\theta \left(S_u - |u|\frac {\nu(\theta_0)} {\theta_0} \right)
\]
with critical parameter $\vartheta = \frac {\theta_0} \theta$. Proposition 3.2 in \cite{above} entails
\[
 n^{\frac 3 2 \vartheta} \sum_{|u|=n} e^{ -V_u} Y_u \xrightarrow[n \rightarrow \infty]{d} Z
\]
where Z is positive almost surely. Furthermore, by equation (1.13) in \cite{above}, we have $Z \stackrel{d}{=} D^{ \frac 1 \vartheta} Z_\vartheta$, where $D$ is the limit of the derivative martingale associated with $-\vartheta V$, and $Z_\vartheta$ is strictly $\vartheta$-stable independent of $D$. If we let $\psi$ be the log-Laplace transform of $- \vartheta V$, then it satisfies the equation $\psi(1) = 0 = \psi'(1)$, so the derivative martingale associated with $- \vartheta V$ is 
\[
\sum_{|u|=n} \vartheta V_u e^{-\vartheta V_u} =  -\sum_{|v|=n} \left(\theta_0 S_v - n \nu(\theta_0) \right) e^{\theta_0 S_v - n\nu(\theta_0)} = D_n
\] 
so $D^{\frac 1 \vartheta} = (D_\infty)^{ \frac \theta {\theta_0}}$. Therefore
\[
n^{\frac {3\theta}{2\theta_0}} Y_n(\theta) e^{n \frac {\theta \nu \left(\theta_0\right)} {\theta_0}} = n^{\frac 3 2 \eta} \sum_{|u|=n} e^{ V_u} Y_u \xrightarrow[n \rightarrow \infty]{d} Z
\]
where the distribution of $Z$ is as in the statement.
By \eqref{coupling} we have 
\begin{align*}
R_n^* - n \frac {\nu(\theta_0)} {\theta_0} + \frac {3\log n} {2 \theta_0} &\stackrel{d}{=} \frac 1 \theta \left(\log Y_n(\theta) - \log E\right) - n \frac {\nu(\theta_0)} {\theta_0} + \frac {3 \log n} {2 \theta_0} \\ &= \frac 1 \theta \left(\log n^{\frac {3\theta}{2\theta_0}} Y_n(\theta) e^{-n \frac {\theta \nu \left(\theta_0\right)} {\theta_0}} - \log E\right) \\& 
\end{align*}
and by  the continuous mapping theorem 
\[
\log n^{\frac {3\theta}{2\theta_0}} Y_n(\theta) e^{-n \frac {\theta \nu \left(\theta_0\right)} {\theta_0}} \xrightarrow[n \rightarrow \infty]{d} \log Z
\]
which completes the proof.
\end{proof}

\section{Proof of Theorems \ref{a.s. convergence below} and \ref{a.s. convergence above}}

We start with the following Lemma, which gives the convergence in probability.

\begin{lemma}\label{convergence in probability}
\item (a) If conditions of Theorem \ref{a.s. convergence below} hold, then
\begin{align*}
 \frac {R_n^*} n \xrightarrow[ n \rightarrow \infty]{\P} \frac {\nu\left(\gamma \theta\right)} {\gamma \theta}
 \end{align*}
\item (b) If conditions of Theorem \ref{a.s. convergence above} hold, then
\begin{equation*}
 \frac {R_n^*} n \xrightarrow[ n \rightarrow \infty]{\P}
 \frac {\nu(\theta_0)} {\theta_0}
\end{equation*}
\end{lemma}

\begin{proof}
Let $\beta = \gamma \theta$ in case (a) and $\beta = \theta_0$ in case (b). We will prove first that
\begin{equation}\label{eq:c3}
\frac 1 {n \theta} \log \big(Y_n(\theta) e^{-n \frac { \theta \nu\left(\beta\right)}{\beta}} \big) \xrightarrow[n \rightarrow \infty]{\mathbb{P}} 0
\end{equation}

We consider first the case (b).  Fix an arbitrary $\varepsilon > 0$ and choose $\delta<\frac {\theta_0} {\theta}$ satisfying 
\[
\frac {\nu(\theta_0)} { \theta_0} -\frac {\nu \left(\delta \theta\right)} {\delta \theta} + \varepsilon > 0.
\] The Markov inequality yields
\begin{align*}
    \mathbb{P}\left( \frac 1 {n \theta} \log\big( Y_n(\theta) e^{-n \frac {\theta \nu \left(\theta_0\right)} {\theta_0} }\big) > \varepsilon\right) 
    &=  \mathbb{P}\left( \delta \log \big(Y_n(\theta) e^{-n \frac {\theta \nu \left(\theta_0\right)} {\theta_0} }\big) > n \theta \delta \varepsilon\right) 
    \\&= \mathbb{P}\left(  Y_n(\theta)^\delta e^{-n \delta \frac {\theta \nu \left(\theta_0\right)} {\theta_0} } > e^{n \theta \delta \varepsilon}\right) 
    \\&\leq \mathbb{E}[Y_n\left(\theta\right)^\delta] e^{- \delta n \theta \left(\frac {\nu(\theta_0)} { \theta_0} + \varepsilon \right)}.
\end{align*}
Since $\delta < 1$ and for any $v$ the random variable $Y_v$ is independent of $S_v$, we obtain
\begin{equation}\label{eq:st}
     \mathbb{E}[Y_n(\theta)^\delta] = \mathbb{E}\bigg[\bigg(\sum_{|v|=n} e^{\theta  S_v} Y_v\bigg)^\delta\bigg] \leq \mathbb{E}\bigg[\sum_{|v|=n} e^{\theta \delta S_v} Y_v^\delta\bigg] = e^{n\nu\left(\theta \delta\right)} \mathbb{E}[Y^\delta],
\end{equation} where the last expectation is finite. Summarizing
\begin{align*}
        \mathbb{P}\left( \frac 1 {n \theta} \log\big( Y_n(\theta) e^{-n \frac {\theta \nu \left(\theta_0\right)} {\theta_0} }\big) > \varepsilon\right) 
        \leq \mathbb{E}[Y^\delta] e^{- \delta n \theta \left(\frac {\nu(\theta_0)} { \theta_0} -\frac {\nu \left(\delta \theta\right)} {\delta \theta} + \varepsilon \right)}
\end{align*}
and thanks to our choice of $\delta$ the above expression converges to 0 as $n$ tends to $+\infty$.
To prove the remaining bound, denote $v_n = \arg \max_{|v| = n} S_v$ . Since
\begin{align*}
    \frac 1 {n \theta} \log\big( Y_n(\theta) e^{-n \frac { \theta \nu(\theta_0)}{\theta_0}}\big) \geq \frac 1 {n \theta} \log \big( e^{n \theta\left( \frac {S_{v_n}} n - \frac {\nu(\theta_0)} {\theta_0} \right)}Y_{v_n} \big) = \frac {R_n} n - \frac {\nu(\theta_0)} {\theta_0} + \frac 1 {n \theta} \log Y_{v_n},
\end{align*}
for any parameters $0 < \delta < \varepsilon$ we obtain
\begin{align*}
\mathbb{P}\left( \frac 1 {n \theta} \log \big(Y_n(\theta) e^{-n \frac {\theta \nu \left(\theta_0\right)} {\theta_0} } \big)< -\varepsilon\right)&\leq \mathbb{P} \left(\frac {R_n} n - \frac {\nu(\theta_0)} {\theta_0} + \frac 1 {n \theta} \log Y_{v_n} < - \varepsilon \right) \\&= \mathbb{P} \left(e^{ n \theta \left( \frac {R_n} n - \frac {\nu(\theta_0)} {\theta_0} \right) }  Y_{v_n} <  e^ {- \varepsilon n \theta} \right) \\&=
\mathbb{P} \left(e^{ n \theta \left( \frac {R_n} n - \frac {\nu(\theta_0)} {\theta_0}  + \varepsilon \right) }Y_{v_n} <  1\right) \\&\leq  
\mathbb{P} \left(e^{ n \theta \delta} Y_{v_n} <  1\right) + \mathbb{P} \left( \frac {R_n} n - \frac {\nu(\theta_0)} {\theta_0}  + \varepsilon <\delta \right)
\end{align*}
Now, since $\frac {R_n} n$ converges almost surely to $ \frac{\nu(\theta_0)} {\theta_0}$ and $\delta < \varepsilon$, the second term converges to 0. For the first term we have
\[
\mathbb{P} \left(e^{ n \theta \delta} Y_{v_n} <  1\right) = \mathbb{P} \left(Y < e^ {- n \theta \delta }\right) \rightarrow 0.
\]
Thus, we conclude the proof of \eqref{eq:c3} for case (b). 

For case (a), by Theorem \ref{convergence in distribution below the boundary}, $\log  Y_n(\theta)e^{-n\frac {\nu\left(\gamma \theta\right)} \gamma}$ converges in distribution to $\log H_\theta$ and this limit is finite almost surely. Therefore $\frac 1 {n \theta} \log \big( Y_n(\theta)e^{-n\frac {\nu\left(\gamma \theta\right)} \gamma})$ converges in distribution to 0, and hence the convergence holds in probability as well. Thus, the proof of \eqref{eq:c3} is completed.

To prove Lemma \ref{convergence in probability} notice  that using \eqref{coupling} we can write 
\begin{align*}
    \frac {R^*_n(\theta)} n &\stackrel{d}{=} \frac {\log Y_n(\theta)} {n \theta} - \frac {\log E } {n \theta} = \frac 1 {n \theta} \log \big( Y_n(\theta)e^{-n\frac {\theta \nu\left(\beta\right)} {\beta}}\big) + \frac {\nu \left( \beta\right)} {\beta} - \frac {\log E } {n \theta} 
\end{align*}
Now $\frac {\log E } {n \theta} $ converges to 0 almost surely and by \eqref{eq:c3}, $\frac 1 {n \theta} \log  Y_n(\theta)e^{-n\frac {\theta \nu\left(\beta\right)} {\beta}}$ converges to 0 in probability. That completes the proof of the Lemma.
\end{proof}

\begin{proof}[Proof of Theorems \ref{a.s. convergence below} and \ref{a.s. convergence above} (almost sure convergence)]
To prove the almost sure convergence we adopt here the arguments given in the proof of Theorem 2.1 in \cite{lpm-brw}.  For the sake of completeness, we present a complete proof. Again let $\beta = \gamma \theta$ if conditions of Theorem \ref{a.s. convergence below} are satisfied and $\beta = \theta_0$ if conditions of Theorem \ref{a.s. convergence above} are satisfied. We start with the upper bound
 \begin{equation}\label{eq:n2}
\limsup\limits_{n \rightarrow \infty}{\frac {R^*_n(\theta)} n  } \leq \frac {\nu\left(\beta\right)} {\beta}
\hspace{0.5cm}\text{ a.s.}
\end{equation}
Fix any $\varepsilon > 0$. By \eqref{coupling} and the Markov inequality we get that for any $\delta < \min(\frac {\theta_0} {\theta}, 1)$
\begin{align*}
    \mathbb{P}\left( \frac {R^*_n(\theta)} n - \frac {\nu\left( \beta\right)} {\beta} > \varepsilon\right) &=
    \mathbb{P}\left( {\theta \delta R^*_n(\theta)} -  \frac {\theta \delta n \nu\left( \beta\right)} {\beta } > n \delta \theta \varepsilon\right)\\ & =
    \mathbb{P}\left( \log \frac {Y_n(\theta)^\delta} {E^\delta} -   \frac {\theta \delta n \nu\left( \beta\right)} {\beta} > n \delta \theta \varepsilon\right) \\ &\leq
    e^{ -\delta n \theta \left(\frac {\nu\left(\beta\right)} {\beta} + \varepsilon \right)} \mathbb{E}\left[E^{-{\delta }}\right] \mathbb{E}\left[Y_n(\theta)^{ \delta }\right] \\ &\le  e^{-  \delta n \theta \left(\frac {\nu\left(\beta\right)} {\beta} - \frac {\nu\left(\theta\delta\right)} {\theta\delta} + \varepsilon \right)} \Gamma\left(1 - \delta \right) \mathbb{E}\left[Y^{ \delta }\right],
\end{align*} where the last inequality follows from \eqref{eq:st}. 

Since $\nu$ is continuous, we can choose $\delta$ so that \[\frac {\nu\left(\beta\right)} {\beta} -\frac {\nu \left(\delta \theta\right)} {\delta \theta} + \varepsilon > 0,\]
therefore the series
\[
\sum_{n=1}^\infty \mathbb{P}\left( \frac {R^*_n(\theta)} n - \frac {\nu\left( \beta\right)} {\beta} > \varepsilon\right)    
\]
converges. The Borel-Cantelli lemma and arbitrariness of $\varepsilon$ entails \eqref{eq:n2}.

Finally our goal is to prove the lower bound

		\begin{equation}\label{eq:n4}
			\liminf_{n\rightarrow\infty} \frac{R_n^*\left(\theta\right)}{n}\geq \frac{\nu\left(\beta\right)}{\beta} \hspace {0.5cm}\text{ a.s.}
\end{equation}

For $u$ such that $|u|=m\leq n$, we define
		\[
		R_{n-m}^{*\left(u\right)}\left(\theta\right):=\max_{v>u, |v|=n} \left(S\left(v\right)+\frac{1}{\theta}\log \left(Y_v/E_v\right)\right)- S\left(u\right).
		\]
		Note that $\{R_{n-m}^{*\left(u\right)}\left(\theta\right)\}_{|u|=m}$ are i.i.d. and have the same distribution as $R_{n-m}^*\left(\theta\right)$. Now,
		\begin{align*}
			R_{n}^*\left(\theta\right)
			&=\max_{|u|=m}\max_{v>u, |v|=n}\left(S\left(v\right)+\frac{1}{\theta}\log \left(Y_v/E_v\right)\right)\\
			&=\max_{|u|=m}\left( S\left(u\right)+R_{n-m}^{*\left(u\right)}\left(\theta\right)\right)\\
			&\geq  S\left(\tilde{u}_m\right)+\max_{|u|=m}\left(R_{n-m}^{*\left(u\right)}\left(\theta\right)\right), 
		\end{align*}	
		where
		\[
		\tilde{u}_m:=\arg\max_{|u|=m}\left(R_{n-m}^{*\left(u\right)}\left(\theta\right)\right).
		\]	
		Now, for any $\varepsilon\in\left(0,1\right)$ and small $s$ such that  $\nu(- {s} /2)$ is finite,
		\begin{align*}
			&\mathbb{P}\left( \frac{R_n^*\left(\theta\right)}{n}-\frac{\nu\left(\beta\right)}{\beta}<-\varepsilon \right)\\
			&\leq\,\,  \mathbb{P}\left(S\left(\tilde{u}_{[\sqrt{n}]}\right)+\max_{|u|=[\sqrt{n}]}\left(R_{n-[\sqrt{n}]}^{*\left(u\right)}\left(\theta\right)\right)<n\left(\frac{\nu\left(\beta\right)}{\beta}-\varepsilon\right)\right) \\
			&\leq \,\,
			\mathbb{P}\left(\max_{|u|=[\sqrt{n}]}\left(R_{n-[\sqrt{n}]}^{*\left(u\right)}\left(\theta\right)\right)<n\left(\frac{\nu\left(\beta\right)}{\beta}-\frac{\varepsilon}{2}\right)\right)
			+
			\mathbb{P}\left(S\left(\tilde{u}_{[\sqrt{n}]}\right)<-\frac{n\varepsilon}{2}\right)\\
			&\leq\,\, 
			\mathbb{E}\left[\mathbb{P}\left(R_{n-[\sqrt{n}]}^*\left(\theta\right)<n\left(\frac{\nu\left(\beta\right)}{\beta}-\frac{\varepsilon}{2}\right)\right)^{N_{[\sqrt{n}]}}\right]+
			e^{-n\varepsilon s/4}\cdot\mathbb{E}\left[e^{-s S\left(\tilde{u}_{[\sqrt{n}]}\right)/2}\right],
		\end{align*}
        where $N_k$ is the number of offspring in $k$-th generation. Recalling Lemma \ref{convergence in probability}, for all large enough $n$,
		\[
		\mathbb{P}\left(R_{n-[\sqrt{n}]}^*\left(\theta\right)<n\left(\frac{\nu\left(\beta\right)}{\beta}-\frac{\varepsilon}{2}\right)\right)<\varepsilon.
		\]

		Observe that $N_k < n$ implies at least $k-\lceil\log_2n\rceil$ many particles up to generation $k$ give birth to only one offspring. Indeed, denote as $A_k$ the number of particles up to $k$-th generation that had only one child. Then clearly 
		\[
		N_k < n \implies k - \lceil \log_2 N_k \rceil \geq k - \lceil \log_2n \rceil
		\]
		so it suffices to show $k - \lceil \log_2 N_k \rceil \leq A_k$. This on the other hand is equivalent to $2^k \leq 2^{\lceil \log_2 N_k \rceil} 2^{A_k}$, which is obviously true if $A_k \geq k$. If $A_k < k$, then it means that there is at least $k - A_k$ generations in which all the particles have at least 2 children, so $N_k \geq 2^{k - A_k}$, which again gives the desired result.
		Therefore 
		\[
		\P\left(N_{[\sqrt{n}]} < n \right)\leq \left(\P\left(N=1\right)\right)^{[\sqrt{n}]-\lceil\log_2n\rceil}.
		\]
		so 
\begin{multline*}
    \mathbb{E}\left[\mathbb{P}\left(R_{n-[\sqrt{n}]}^*\left(\theta\right)<n\left(\frac{\nu\left(\beta\right)}{\beta}-\frac{\varepsilon}{2}\right)\right)^{N_{[\sqrt{n}]}}\right] \leq \E [\varepsilon^{N_{[\sqrt{n}]}}] \\\leq \E\left[\varepsilon^n ; N_{[\sqrt{n}]} \geq n\right] + \P \left(N_{[\sqrt{n}]} < n\right) \leq \varepsilon^n+\left(\P\left(N=1\right)\right)^{[\sqrt{n}]-\lceil\log_2n\rceil} \le \varepsilon_1^{\sqrt n}
\end{multline*} for some $\varepsilon_1<1$. 
To estimate the second term, we bound  supremum by the sum and we have
		\[
		\mathbb{E}\left[e^{-s S\left(\tilde{u}_{[\sqrt{n}]}\right)/2}\right]
		\leq \mathbb{E}\bigg[\sum_{|v|= \lceil \sqrt{n} \rceil}  e^{- \frac s 2 S_v}\bigg]
		=e^{[\sqrt{n}]\nu\left(-s/2\right)}.
		\]
		Therefore we have for all large enough $n$,
		\[
			\mathbb{P}\left( \frac{R_n^*\left(\theta\right)}{n}-\frac{\nu(\beta)}{\beta}<-\varepsilon \right)
			\leq \, \varepsilon_1^{\sqrt n} +e^{-n\varepsilon s/4+[\sqrt{n}]\nu\left(-s/2\right)}.
	\]
		Since for every $\varepsilon\in\left(0,1\right)$,
		\begin{equation*}
			\label{bc-}
			\sum_{n=1}^{\infty} \P\left( \frac{R_n^*\left(\theta\right)}{n}-\frac{\nu(\beta)}{\beta}<-\varepsilon \right)<\infty,
		\end{equation*}
		using Borel-Cantelli Lemma once again we deduce \eqref{eq:n4}, completing the proof.
\end{proof}

\end{document}